\documentclass[a4paper,12pt]{article}
\usepackage[danish,english,]{babel}
\usepackage[T1]{fontenc}
\usepackage[utf8]{inputenc}
\usepackage{lmodern}

\usepackage{amsthm,amsmath,amsfonts}
\DeclareMathOperator{\Var}{Var}

\usepackage{enumitem}
\usepackage{mathtools}
\usepackage{bbm}
\usepackage{prodint}

\newcounter{maincounter}
\newtheorem{theorem}[maincounter]{Theorem}
\newtheorem{prop}[maincounter]{Proposition}

\theoremstyle{definition}
\newtheorem{example}[maincounter]{Example}

\newcommand{\given}{\mathbin{|}}
\newcommand{\R}{\mathbb{R}}
\newcommand{\N}{\mathbb{N}}
\renewcommand{\d}{\hspace{1pt} \mathrm{d}}

\renewcommand{\P}{\operatorname{P}}
\newcommand{\E}{\operatorname{E}}
\newcommand{\W}{\mathcal{W}}
\newcommand{\from}{\colon}
\newcommand{\banach}[1]{\mathbf{#1}}
\newcommand{\sigAlg}[1]{\mathcal{#1}}
\newcommand\independent{\protect\mathpalette{\protect\independenT}{\perp}}
\def\independenT#1#2{\mathrel{\rlap{$#1#2$}\mkern2mu{#1#2}}}
\renewcommand{\c}{^{\mathsf{c}}}
\newcommand{\rc}{^{\mathsf{r}}}
\newcommand{\indic}[1]{\mathbf{1}(#1)}

\title{Counting processes in $p$-variation with applications to recurrent events}
\author{
Morten Overgaard \\
\small Department of Public Health, Aarhus University, \\
\small Bartholins All\'{e} 2, DK-8000 Aarhus C, Denmark \\
\small moov@ph.au.dk
}

\begin{document}
\maketitle

\begin{abstract}
Convergence results for averages of independent replications of counting processes are established in a $p$-variation setting and under certain assumptions. Such convergence results can be combined with functional differentiability results in $p$-variation in order to study the asymptotic properties of estimators that can be considered functionals of such averages of counting processes. Examples of this are given in recurrent events settings, confirming known results while also establishing the appropriateness of the pseudo-observation method for regression analysis. In a recurrent events setting with a terminal event, it is also established that it is more efficient to discard complete information on a censoring time and instead consider the censoring times censored by the terminal event.
\end{abstract}

\section{Introduction}

The concept of $p$-variation allows for an elegant way of studying asymptotic properties of estimators depending on the data through the empirical distribution function of one or more one-dimensional variables. 
This is because of two things.
Firstly, the empirical distribution function, based on independent and identically distributed observations, converges to the true distribution function in $p$-variation for many $p$. 
Secondly, many such estimators may be described as differentiable functionals of the empirical distribution function in a $p$-variation setting, at least for some values of~$p$.
This means that asymptotic properties can be derived by what is essentially a functional delta method.
Unfortunately, many estimators do not depend only on the empirical distribution of one-dimensional variables, but will rather be more complex.
This motivates looking for extensions to this approach, which involves proving convergence in $p$-variation of more general averages to the true mean.  

In this paper, we will see how the approach described above can be used for estimators based on averages of counting processes.
The main result is a convergence result for counting processes in $p$-variation which builds on and extends results by Qian \cite{qian1998}.
Applications to different estimators of the mean function in a recurrent events setting serve as demonstrations of how the described approach can be used to derive asymptotic properties of the different estimators. Expressions of influence functions and asymptotic variances and suggestions for natural variance estimators are also given for the different estimators and the appropriateness of the pseudo-observation method based on these estimators is established under some conditions.

\section{Main results} \label{sec:main}

A counting process $N$ is characterized by taking only non-negative integer values, $N \in \N_0$, and being increasing, $N(t) \geq N(s)$ for $t \geq s$.
In our setting, counting processes are also continuous on the right with limits on the left and are defined on the time interval $[0, \infty]$ with the requirements $N(0) = 0$ and $N(\infty) = \lim_{t \to \infty} N(t)$.
A counting process which is no larger than 1 is called simple.

For a counting process $N$, the time points $T_k := \inf\{s : N(s) \geq k\}$ for $k \in \N$ characterize the process. Specifically, we can define simple counting processes by $N^{(k)}(s) = \indic{T_k \leq s}$, $s \in [0, \infty)$ for $k \in \N$, in which case 
\begin{equation} \label{eq:counting_decomposition}
    N(s) = \sum_{k = 1}^\infty N^{(k)}(s)
\end{equation}
for all $s \in [0, \infty]$. 

If $E(N(\infty)) < \infty$ for a counting process, $N$, as defined above then $\E(N(s)) < \infty$ for all $s \in [0,\infty]$. In this case, the notation $F$ for the mean function given by $F(s) = \E(N(s))$ will be used in the following.

Consider $p \geq 1$. For a real function $f$ defined on an interval $J$ the $p$-variation is defined by
\begin{equation}
    v_p(f;J) = \sup_{m \in \N, t_0, \dots, t_m \in J} \sum_{j=1}^m |f(t_j)-f(t_{j-1})|^p
\end{equation}
and $f$ is of bounded $p$-variation when $v_p(f;J) < \infty$. An equivalent definition and the following results can be found in the book by Dudley and Norvai\v{s}a \cite{Dudley2010}.
The definition $\|f\|_{(p)} = v_p(f;J)^{\frac{1}{p}}$ leads to a seminorm and if $\|f\|_\infty = \sup_{t \in J} |f(t)|$, norms are defined by both $\| f \|_\infty$, the supremum norm, and $\| f\|_{[p]} = \|f\|_{(p)} + \|f\|_{\infty}$, the $p$-variation norm.
The space $\W_p(J)$ of all functions $f \from J \to \R$ of bounded $p$-variation, $v_p(f;J) < \infty$, is a Banach space, a complete normed vector space, when equipped with the $p$-variation norm.
This is similarly the case for the subset $\W_p\rc(J)$ of right-continuous functions.
The interval $J$, which will be $[0,\infty)$ in the following, is generally dropped from the notation when it is clear from the context as has already been done in the cases of $\|\cdot\|_{\infty}$, $\|\cdot\|_{(p)}$, and $\|\cdot\|_{[p]}$.
An important convergence result by Qian, Theorem~3.2 of \cite{qian1998}, implies the following result. 
\begin{prop} \label{prop:QianE}
Let $F$ be a cumulative distribution function with mass on $(0, \infty)$ and let, for each $n \in \N$, $F_n$ be the empirical distribution of $n$ independent observations $X_1, \dots, X_n \in [0,\infty)$ from the distribution given by $F$, that is, given by $F_n(t) = \frac{1}{n} \sum_{i=1}^n \indic{X_i \leq t}$. Then, for $p \in [1, 2)$, a constant $C_p$, only depending on~$p$, exists such that
\begin{equation} \label{eq:QianE}
    \E(v_p(F_n - F)) \leq C_p n^{1-p}
\end{equation}
for any $n \in \N$.
\end{prop}
The inequality of \eqref{eq:QianE} also implies 
\begin{equation} \label{eq:QianE_alt}
    \E(\|F_n - F\|_{[p]}) \leq 2C_p^{\frac{1}{p}} n^{\frac{1-p}{p}}
\end{equation}
according to Jensen's inequality with the convex function $x \mapsto x^p$ and the fact that $F_n(0)-F(0)=0$ under the same conditions.

The processes $t \mapsto \indic{X_i \leq t}$ at play above are simple counting processes. 
Proposition~\ref{prop:QianE} is the main ingredient in obtaining the following convergence result for more general simple counting processes, which generalizes~\eqref{eq:QianE_alt}.
\begin{theorem} \label{theorem:simple}
Consider a simple counting process $N$ with mean function $F$. For each $n \in \N$, let $F_n = n^{-1} \sum_{i=1}^n N_{i}$ be the average of independent replications $N_1, \dots, N_n$ of $N$. Then, for any $p \in [1, 2)$ a constant $K_p$, only depending on $p$, exists such that 
\begin{equation}
\E(\| F_n - F \|_{[p]}) \leq K_p F(\infty)^{\frac{1}{p}} n^{\frac{1-p}{p}}
\end{equation}
for any $n \in \N$.
\end{theorem}

\begin{proof} 
Let $p \in [1,2)$, $n \in \N$ and the processes $N_1, \dots, N_n$ be given and let $\tilde n = \# \{i : N_i(\infty) = 1 \}$ be the number of observed jumps. Interpreting division by 0 as 0 allows for the equations
\begin{equation} \label{eq:FnFsplit}
\begin{aligned}
    F_n - F &= \frac{\tilde n}{n} \frac{1}{\tilde n}\sum_{i=1}^n N_i - F(\infty) \frac{F}{F(\infty)} \\
    &= \frac{\tilde n}{n} \big( \frac{1}{\tilde n} \sum_{i=1}^n N_i - \frac{F}{F(\infty)} \big) + ( \frac{\tilde n}{n} - F(\infty)) \frac{F}{F(\infty)}.
\end{aligned}
\end{equation}
Because $F$ is 0 at 0 and increasing, we have $\| F /F(\infty)\|_{[p]} \leq 2$. Note that $F(\infty) \in [0,1]$ because $N$ is simple. 
Since $\tilde n$ follows a binomial distribution of $n$ trials with probability $F(\infty)$, we have $\E(| \tilde n - n F(\infty)|) \leq \E(\tilde n(1-F(\infty)) + (n-\tilde n) F(\infty)) = 2 n F(\infty)(1-F(\infty))$ and $\E(| \tilde n - n F(\infty)|^2) = \Var(\tilde n) = nF(\infty)(1-F(\infty))$ and so $\E(| \tilde n - n F(\infty)|^p) \leq \E(\max(| \tilde n - n F(\infty)|, | \tilde n - n F(\infty)|^2)) \leq \E(| \tilde n - n F(\infty)|) + \E(| \tilde n - n F(\infty)|^2) \leq 3 n F(\infty)(1-F(\infty))$. It is then an application of Jensen's inequality with the convex function $x \mapsto x^p$ which reveals
\begin{equation}
\begin{aligned}
    \E \big(\big| \frac{\tilde n}{n} - F(\infty) \big|\big) &\leq \frac{1}{n} \E(| \tilde n - n F(\infty)|^p)^{\frac{1}{p}} \leq \frac{1}{n} \big(3n F(\infty)(1-F(\infty))\big)^{\frac{1}{p}} \\
    & \leq 3^{\frac{1}{p}} n^{\frac{1-p}{p}} F(\infty)^{\frac{1}{p}}.
\end{aligned}
\end{equation}
The function $F/F(\infty)$ is a cumulative distribution function. In the conditional distribution given $N_1(\infty), \dots, N_n(\infty)$, the $\tilde n^{-1} \sum_{i=1}^n N_i$ is an empirical distribution of $\tilde n$ independent observations from the distribution given by $F/F(\infty)$. 
Let $\sigAlg{A}$ be the $\sigma$-algebra generated by $N_1(\infty), \dots, N_n(\infty)$.
An application of Proposition~\ref{prop:QianE} and Jensen's inequality as before in the conditional distribution reveals, almost surely,
\begin{equation}
\begin{aligned}
    \E \big(\big\|\frac{1}{\tilde n} \sum_{i=1}^n N_i - \frac{F}{F(\infty)} \big\|_{[p]} \given \sigAlg{A} \big) &\leq 2\E \big(v_p \big(\frac{1}{\tilde n} \sum_{i=1}^n N_i - \frac{F}{F(\infty)} \big) \given \sigAlg{A} \big)^{\frac{1}{p}} \\
    &\leq 2C_p^{\frac{1}{p}} \tilde n^{\frac{1-p}{p}},
\end{aligned}
\end{equation}
where $C_p$ is the constant from Proposition~\ref{prop:QianE}.
It is worth noting that $\sigma(\tilde n) \subseteq \sigAlg{A}$ and that the same constant $C_p$ can be used no matter the distribution according to Proposition~\ref{prop:QianE}.
The considerations above mean that equation~\eqref{eq:FnFsplit} leads to
\begin{equation}
\begin{aligned}
    \E(\|F_n - F\|_{[p]}) &\leq 2C_p^{\frac{1}{p}} \frac{1}{n} \E(\tilde n^{\frac{1}{p}}) +  3^{\frac{1}{p}} 2 n^{\frac{1-p}{p}} F(\infty)^{\frac{1}{p}} \\
    &\leq 2(C_p^{\frac{1}{p}} + 3^{\frac{1}{p}}) n^{\frac{1-p}{p}} F(\infty)^{\frac{1}{p}}
\end{aligned}
\end{equation}
where Jensen's inequality is again used to establish that $\E((\tilde n / n)^{1/p}) \leq \E(\tilde n/ n)^{1/p} = F(\infty)^{1/p}$. This shows the desired upper bound with the constant $K_p = 2(C_p^{\frac{1}{p}} + 3^{\frac{1}{p}})$.
\end{proof}

Let us now turn our attention to more general counting processes as defined in the beginning of this section. The characterization of a counting process as a sum of simple counting processes allows us to establish the following convergence result by appealing to Theorem~\ref{theorem:simple}.

\begin{theorem} \label{theorem:generalE}
Let $p \in [1,2)$ be given and consider a counting process $N$ with mean function $F$ such that $N(\infty)$ has finite moment of order $p+\varepsilon$ for some $\varepsilon > 0$. 
For each $n \in \N$, let $F_n = n^{-1} \sum_{i=1}^n N_{i}$ be the average of independent replications $N_1, \dots, N_n$ of $N$.
Then a constant $C$ exists, depending on $p$ as well as the distribution of $N(\infty)$, such that
\begin{equation} \label{eq:multjumpE}
\E(\| F_n - F \|_{[p]}) \leq C n^{\frac{1-p}{p}},
\end{equation}
for any $n \in \N$.
\end{theorem}

\begin{proof} 
Let $n \in \N$ and the processes $N_1, \dots, N_n$ be given.
For each $k \in \N$, let $N^{(k)}$ denote the simple counting process corresponding to the decomposition of $N$ as in \eqref{eq:counting_decomposition} and similarly with $N_i^{(k)}$ for $i = 1, \dots, n$. Also, for $k \in \N$, let $F^{(k)}$ denote the mean function of $N^{(k)}$ and let $F_n^{(k)} = n^{-1} \sum_{i=1}^n N_i^{(k)}$ be the empirical mean function. We then obtain the identities $F = \sum_{k=1}^\infty F^{(k)}$ and $F_n = \sum_{k=1}^\infty F_n^{(k)}$. 
By the triangle inequality, the monotone convergence theorem, and Theorem~\ref{theorem:simple}, we now have
\begin{equation} \label{eq:FnF_E_bound_general}
    \E(\|F_n - F\|_{[p]}) \leq \sum_{k=1}^\infty \E(\|F_n^{(k)} - F^{(k)}\|_{[p]}) \leq K_p \sum_{k=1}^\infty F^{(k)}(\infty)^{\frac{1}{p}} n^{\frac{1-p}{p}}.
\end{equation}
Now, $F^{(k)}(\infty) = \P(N(\infty) \geq k)$ and by Markov's inequality, we have, for the given $\varepsilon > 0$, $P(N(\infty) \geq k) \leq \E(N(\infty)^{p+\varepsilon})/k^{p+\varepsilon}$. Since $\E(N(\infty)^{p+\varepsilon}) < \infty$ by the moment condition, then
\begin{equation}
    C:= K_p \sum_{k=1}^\infty F^{(k)}(\infty)^{\frac{1}{p}} \leq K_p \E(N(\infty)^{p+\varepsilon})^{\frac{1}{p}} \sum_{k=1}^{\infty}  k^{-1-\frac{\varepsilon}{p}} < \infty
\end{equation}
and this proves the desired result.
\end{proof}

The bound in expectation in \eqref{eq:multjumpE} immediately, by Markov's inequality, gives the useful bound in probability,
\begin{equation}
\| F_n - F \|_{[p]} = O_{\P}(n^{\frac{1-p}{p}})
\end{equation}
under the same assumptions.
The following result gives an almost sure bound in some cases and has its inspiration in Theorem~4.2 of \cite{qian1998}. The proof will rely heavily on a lemma from \cite{dudley1983invariance} in addition to the results of Theorem~\ref{theorem:generalE} above.
\begin{theorem} \label{theorem:almost_sure}
Consider a counting process $N$ such that $N(\infty) \leq B$ almost surely for some $B > 0$. 
For each $n \in \N$, let $F_n = n^{-1} \sum_{i=1}^n N_{i}$ be the average of independent replications $N_1, \dots, N_n$ of $N$.
Then for any $p \in [1,2)$ a constant $\lambda$ exists, depending on $p$ and $B$, such that 
\begin{equation} \label{eq:as_bound}
\limsup_{n \to \infty} n^{\frac{p-1}{p}} \| F_n - F \|_{[p]}  \leq \lambda 
\end{equation}
almost surely. In particular, $\|F_n - F\|_{p]} = O(n^{(1-p)/p})$ almost surely in this case.
\end{theorem}
\begin{proof}
The statement is trivial for $p=1$ where $\|F_n - F\|_{[p]} \leq \lambda$ for all $n$ for $\lambda = 4B$. Consider a given $p \in (1,2)$ and, for now, a given $n \in \N$ and the processes $N_1, \dots, N_n$. 
We let $X_j = N_j -F$, for $j=1, \dots, n$, as well as $S_n = \sum_{j=1}^n X_j$ be random elements of $\W_p$. The $\{X_j\}_{j=1}^n$ can be considered an independent sequence in the terminology of Section~2 of \cite{dudley1983invariance}. Owing to the upper bound of $N$ by $B$, we have $F(s) \leq B$ and so $\|X_j\|_{[p]} \leq 4B =: M$ and $\sum_{j=1}^n \E(\|X_j\|_{[p]}^2) \leq n (4B)^{2} =: \tau_n$. Lemma~2.6 of \cite{dudley1983invariance} now states that 
\begin{equation} \label{eq:dudley_bound}
    \P(\|S_n\|_{[p]} \geq K) \leq \exp(3 \gamma^2 \tau_n - \gamma (K - \E(\|S_n\|_{[p]})))
\end{equation}
for any $\gamma \in [0, (2M)^{-1}]$ and any $K > 0$. 
Now, supposing $n$ is sufficiently large that $n^{(1-p)/p} \leq (2M)^{-1}$, we want to use this with $\gamma = n^{(1-p)/p}$ and $K = t n^{1/p}$ for a given $t > 0$. Note that $S_n = n(F_n - F)$ and use that $\E(\|S_n\|_{[p]}) \leq C n^{1/p}$ for some $C > 0$ according to Theorem~\ref{theorem:generalE} to see that equation~\eqref{eq:dudley_bound} implies
\begin{equation} \label{eq:tail_DP}
\begin{aligned}
    \P(n^{\frac{p-1}{p}}\|F_n - F\|_{[p]} \geq t) &\leq \exp \big( 3 n^{2\frac{1-p}{p}} n(4B)^2 - n^{\frac{1-p}{p}}(t n^{\frac{1}{p}} - C n^{\frac{1}{p}})\big) \\
    &= \exp\big(-n^{\frac{2-p}{p}}(t - (C+3(4B)^2))\big)
\end{aligned}
\end{equation}
for any $t > 0$ for sufficiently large $n$. Let $\lambda = C+3(4B)^2$. Since $p < 2$ and if $t > \lambda$ is considered, the tail probability from~\eqref{eq:tail_DP} vanishes rapidly as $n$ increases. In particular, $\sum_{n=1}^\infty \P(n^{(p-1)/p} \|F_n - F\|_{[p]} \geq t)$ converges for $t > \lambda$. The Borel--Cantelli lemma then reveals that $\P( \limsup_{n \to \infty} \{n^{(p-1)/p} \|F_n - F\|_{[p]} \geq t\}) = 0$ for $t > \lambda$. This implies $\limsup_{n \to \infty} n^{(p-1)/p} \|F_n - F\|_{[p]} \leq \lambda$ almost surely, which is the desired result. Looking at the proof of Theorem~\ref{theorem:generalE} and equation~\eqref{eq:FnF_E_bound_general} in particular reveals $C \leq K_p B$, for $K_p$ from Theorem~\ref{theorem:simple}, can be used in this case such that $\lambda$ can be taken to only depend on $p$ and $B$ if desired. The statement $\|F_n - F\|_{[p]} = O(n^{(1-p)/p})$ almost surely means exactly the existence of a $\lambda$ such that \eqref{eq:as_bound} holds almost surely.
\end{proof}

\section{Application in a recurrent events setting}
\label{sec:count}

An example of a counting process is a process counting the number of recurrent events a study participant has experienced. 
Estimation of targets such as the expected number of events by a certain time point may be estimated in a straightforward manner by an average over independent replications when the process is completely observed.
When the counting of events of interest is sometimes prevented by censoring, estimation may be more complicated. 
When the censoring time is itself censored, perhaps by a terminal event, then estimation is further complicated.

In this section, we will see how the convergence results of Section~\ref{sec:main} may be applied to study the asymptotic properties of estimators in this setting by appealing to differentiability properties of the involved estimating functionals.
Here, differentiability means Fréchet differentiability. Appendix~\ref{appendix:differentiability} includes the most important definitions and properties for our purposes in a general Banach space-based setting, primarily based on Chapter~5 of \cite{Dudley2010}. As mentioned in Section~\ref{sec:main}, the function space $\W_p$ of functions of bounded $p$-variation is a Banach space when equipped with the $p$-variation norm, as is the case for the subspace $\W_p\rc$ of right-continuous functions of bounded $p$-variation. In particular owing to the inequalities $\|f g\|_{[p]} \leq \|f\|_{[p]} \|g\|_{[p]}$ and $\|\int_0^{(\cdot)} g(s) f(\d s)\|_{[p]} \leq k_p \|f\|_{[p]} \|g\|_{[p]}$ for a constant $k_p > 0$ for $f,g \in \W_p$ for $p \in [1,2)$, many important functionals are differentiable as functionals between $p$-variation-based spaces. In addition to the two implied bilinear functionals, these differentiable functionals include mapping to the inverse element and product integration.
Above, $\int_0^{(\cdot)} g(s) f(\d s)$ for $f,g \in \W_p$ should be considered a Young integral, see for instance \cite{Dudley1992}. The Young integral does correspond to the Lebesgue--Stieltjes integral when $f$ is of bounded variation, $f \in \W_1$, here. 
More details on these topics can be found in \cite{Dudley2010} and \cite{Dudley1999}. The supplements to the papers \cite{overgaard2017asymptotic, Overgaard2019} include some important details in a more condensed form.

Let $N$ be the counting process of interest, which is assumed square integrable. In this section, we let $\mu$ denote the mean function of $N$ such that $\mu(s) = \E(N(s))$, reserving the notation $F$ for a collection of such means. The mean function $\mu$ is the target of estimation in this section.

\begin{example} \label{example:uncens}
If information is available on $N_1, \dots, N_n$ which are $n$ independent replications of $N$, estimation may be performed by simply taking the average, $\hat \mu_n(t) = n^{-1} \sum_{i=1}^n N_i(t)$.
In this case, $\sqrt{n}(\hat \mu_n(t) - \mu(t))$ has an asymptotic normal distribution with variance $\Var(N(t))$ according to the central limit theorem. 
The convergence result of Theorem~\ref{theorem:generalE} implies that $\| \hat \mu_n - \mu \|_{[p]} = O_{\P}(n^{(1-p)/p})$ for $p \in [1,2)$ and also opens up for use of the functional delta method.
\end{example}

\begin{example} \label{example:cens_obs}
If observation of events is prevented after a right-censoring time $C$, we do not generally have information on the entire $N$, but only on $\tilde N$ given by 
\begin{equation}
    \tilde N(s) = \int_0^s \indic{C \geq u} N(\d u).
\end{equation}
Here, $\tilde N$ is simply another counting process.
Assume that $N$ is independent of $C$ and let $K(s) = \P(C \geq s)$. If $\nu$ denotes the mean function of $\tilde N$, then 
\begin{equation}
    \nu(s) = \int_0^s K(u) \mu(\d u),
\end{equation}
and so, for $s$ such that $K(s) > 0$,
\begin{equation} \label{eq:ipcw_motivation}
    \mu(s) = \int_0^s \frac{1}{K(u)} \nu(\d u).
\end{equation}
Let $X = (\tilde N, C)$ and suppose information is available on $X_1, \dots, X_n$ which are independent replications of $X$ with $X_i = (\tilde N_i, C_i)$. Then $\nu(s)$ can be estimated as in Example~\ref{example:uncens} by $\hat \nu_n(s) = n^{-1} \sum_{i=1}^n \tilde N_i(s)$, and $K(s)$ can be estimated by $\hat K_n(s) = n^{-1} \sum_{i=1}^n \indic{C_i \geq s}$. Equation~\eqref{eq:ipcw_motivation} now suggests the estimate 
\begin{equation} \label{eq:ipcw_estimate_emp}
    \hat \mu_n(s) = \int_0^s \frac{1}{\hat K_n(u)} \hat \nu_n(\d u)
\end{equation}
of $\mu(s)$. This corresponds to the estimator from~(2.2) of \cite{lawless1995some}.

The estimate of \eqref{eq:ipcw_estimate_emp} relies on empirical means of two counting processes, namely $\tilde N$ and $N_C$ given by $N_C(s) = \indic{C \leq s}$. The estimate is in fact obtained from the empirical means of those counting processes by a functional which is differentiable of any order in a $p$-variation setting.
Consider a given $t > 0$ such that $K(t) > 0$. Interest will now be in properties of $\hat \mu_n(s)$ from \eqref{eq:ipcw_estimate_emp} for $s \in [0,t]$. 
This will be studied through a functional approach based on $p$-variation.
To be specific, consider the Banach space $\banach{F} = \W_p\rc([0, \infty))^2$ for a $p \in [1,2)$, with a general element $f = (f_1, f_2) \in \banach{F}$ and norm given by $\|f\|_{\banach{F}} = \max(\|f_1\|_{[p]}, \|f_2\|_{[p]})$. The functional given by $K(f) = (s \mapsto K(f;s))$ with $K(f;s) = f_2(\infty) - f_2(s-)$ is linear and continuous as a functional from $\banach{F}$ to $\W_p([0,t])$. It is therefore differentiable of any order with a first order derivative at $f \in \banach{F}$ in direction $g \in \banach{F}$ which is $K_f'(g) = K(g)$. As can be seen from Theorem~4.16 of \cite{Dudley2010} since $\W_p([0,t])$ is a unital Banach algebra, if $U \subseteq \W_p([0,t])$ is the open subset of $\W_p([0,t])$ of $f$s for which $1/f \in \W_p([0,t])$, then $f \mapsto 1/f$ is differentiable of any order with a first order derivative at $f \in U$ in direction $g \in \W_p([0,t])$ which is $-g/f^2$.
Also, $(f_1,f_2) \mapsto \int_0^{(\cdot)} f_2(s) f_1(\d s)$ is, as a functional from $\W_p([0,t])^2$ to $\W_p([0,t])$, bilinear and continuous and thus differentiable of any order with first order derivative given by $\int_0^{(\cdot)} f_2(s) g_1(\d s) + \int_0^{(\cdot)} g_2(s) f_1(\d s)$. 
An $f \in U$ is characterized by being bounded uniformly away from 0.
It is now the chain rule, see Appendix~\ref{appendix:differentiability}, which reveals that 
\begin{equation}
    \mu \from f \mapsto \int_0^{(\cdot)} \frac{1}{K(f;s)} f_1(\d s)
\end{equation}
is differentiable of any order, at least in a neighborhood of an $f$ such that $K(f)$ is bounded uniformly away from 0, as a functional from $\banach{F}$ into $\W_p\rc([0,t])$. For $f$ where the functional $\mu$ is differentiable, the derivative at $f$ in direction $g$ is, by the chain rule, see~\eqref{eq:chain_rule} of Appendix~\ref{appendix:differentiability},
\begin{equation}
    \mu_f'(g) = \int_0^{(\cdot)} \frac{1}{K(f;s)} g_1(\d s) - \int_0^{(\cdot)} \frac{K(g;s)}{K(f;s)^2} f_1(\d s).
\end{equation}
If we let $G(s) = \P(C \leq s)$ and $F = (\nu, G)$, then the functional $\mu$ is, in particular, differentiable of any order in a neighborhood of $F \in \banach{F}$.
With the notation $x = (\tilde n, c)$ for $\tilde n \in \W_p\rc([0, \infty))$ and $c > 0$, $N_c(s) = \indic{c \leq s}$, and $\delta_x \in \banach{F}$ given by $\delta_x(s) = (\tilde n(s), N_c(s))$, the empirical mean of the counting processes is $F_n = n^{-1} \sum_{i=1}^n \delta_{X_i}$ and we see that $\hat \mu_n$ from \eqref{eq:ipcw_estimate_emp} is obtained by $\hat \mu_n = \mu(F_n)$.
The influence function, defined by $\dot \mu(x) = \mu_F'(\delta_x - F)$, of this estimator can be expressed as
\begin{equation}
    \dot \mu(x) = \int_0^{(\cdot)} \frac{1}{K(s)} \tilde n(\d s) - \int_0^{(\cdot)} \frac{\indic{c \geq s}}{K(s)} \mu(\d s)
\end{equation}
by using that $K(F;s) = K(s)$ and $\mu(s) = \int_0^s K(u)^{-1} \nu(\d u)$.
The differentiability of any order of $\mu$ in a neighborhood of $F$ is enough to establish
\begin{equation} \label{eq:taylor1}
    \mu(F_n) = \mu(F) + \mu_F'(F_n - F) + O(\|F_n - F\|_{\banach{F}}^2)
\end{equation}
as in \eqref{eq:Taylor1+Lip} of Appendix~\ref{appendix:differentiability}.
We have already assumed square integrability of $N$ and thus of $\tilde N \leq N$ while this is trivially the case for the bounded counting process $N_C$. This means that Theorem~\ref{theorem:generalE} ensures $\|F_n - F\|_{\banach{F}} = O_{\P}(n^{(1-p)/p})$. Take $p \in (4/3, 2)$, then we have, in particular, $\|F_n - F\|_{\banach{F}} = o_{\P}(n^{-1/4})$. From \eqref{eq:taylor1} and linearity of $\mu_F'$ we obtain
\begin{equation}
    \sqrt{n}(\hat \mu_n - \mu) = \sqrt{n}\frac{1}{n} \sum_{i=1}^n \dot \mu(X_i) + o_{\P}(1)
\end{equation}
in $p$-variation.
Evaluating at $s \in [0,t]$, this ensures that $\sqrt{n}(\hat \mu_n(s) - \mu(s))$ has the same asymptotic distribution as $n^{-1/2} \sum_{i=1}^n \dot \mu(X;s)$ which is a normal distribution with mean $\E(\dot \mu(X;s)) = 0$ and variance $\Var(\dot \mu(X;s))$. From the alternative expression of the influence function as
\begin{equation} \label{eq:C_obs_infl}
    \dot \mu(X;s) = N(s) - \mu(s) - \int_0^{s-} \frac{(N(s) - \mu(s) - (N(u) - \mu(u)))}{K(u+)} M_C(\d u),
\end{equation}
where $M_C(s) = N_C(s) - \int_0^s \indic{C \geq u} \Lambda(\d u)$, the variance can be seen to be
\begin{equation} \label{eq:var_expr_C_obs}
   \Var(\dot \mu(X;s)) = \Var(N(s)) + \int_0^{s-} \Var(N(s) - N(u)) \frac{1}{K(u+)} \Lambda(\d u)
\end{equation}
under the independence assumption $N \independent C$.
This variance can be estimated by $n^{-1} \sum_{i=1}^n \mu_{F_n}'(\delta_{X_i} - F_n;s)^2$ which turns out to be the variance estimate of (2.3) from \cite{lawless1995some}.
Some more details on the derivations of~\eqref{eq:C_obs_infl} and~\eqref{eq:var_expr_C_obs} can be found in Appendix~\ref{appendix:infl_var}. In comparison to Example~\ref{example:uncens} with no censoring, the last term of~\eqref{eq:var_expr_C_obs} can be seen as the added variance due to censoring when using the estimator $\hat \mu_n$ from \eqref{eq:ipcw_estimate_emp}. 
\end{example}

\begin{example} \label{example:cens_unobs}
Consider the setting of Example~\ref{example:cens_obs} with $N$ censored by $C$, leaving $\tilde N(s) = \int_0^s \indic{C \geq u} N(\d u)$ observed. As mentioned in the beginning of the section, $C$ may itself be censored such that the empirical estimate of $K(s) = \P(C \geq s)$ is not generally available. This is the setting considered in this example.
Suppose a terminal event at time $T$ right-censors observation of $C$. 
Here, $T$ is terminal in the sense that $N(s) = N(T \wedge s)$ for all $s$. We let $\tilde C = C \wedge T$ and $\tilde D = \indic{C < T}$. 
The function $K$ is the left-continuous version of a survival function for $C$ and takes the form $K(s) = \prodi_0^{s-}(1-\Lambda(\d u))$ for a right-continuous cumulative censoring hazard function $\Lambda$. 
If we let $G(s) = \P(C \leq s)$, we have $\Lambda(s) = \int_0^s K(u)^{-1} G(\d u)$. 
We will assume independence of $C$ and $(N,T)$. Then we also have 
\begin{equation} \label{eq:Lambda_identified}
    \Lambda(s) = \int_0^s \frac{1}{K\c(u)} G_1\c(\d u),
\end{equation}
for $s$ such that $K\c(s) > 0$, where $K\c(s) = \P(\tilde C > s) + \P(\tilde C = s, \tilde D = 1)$ and $G_1\c(s) = \P(\tilde C \leq s, \tilde D = 1)$ since, owing to the independence of $C$ and $T$, $K\c(s) = K(s) \P(T > s)$ and $G_1\c(s) = \int_0^s \P(T > u) G(\d u)$.
In this example, the basic observation is $X = (\tilde N, \tilde C, \tilde D)$.
If information is available on $X_1, \dots, X_n$ which are independent replications of $X = (\tilde N, \tilde C, \tilde D)$ with $X_i =(\tilde N_i, \tilde C_i, \tilde D_i)$, we may estimate $K\c(s)$ by $\hat K_n\c(s) = n^{-1} \sum_{i=1}^n (\indic{\tilde C_i > s} + \indic{\tilde C_i=s, \tilde D_i = 1})$ and $G_1\c(s)$ by $\hat G_{n,1}\c(s) = n^{-1} \sum_{i=1}^n \indic{\tilde C_i \leq s, \tilde D_i = 1}$. Equation~\eqref{eq:Lambda_identified} suggests the estimate 
\begin{equation}
    \hat \Lambda_n(s) = \int_0^s \frac{1}{\hat K_n\c(u)} \hat G_{n,1}\c(\d u)
\end{equation}
of $\Lambda(s)$. This estimate then leads to the estimate of $K$ given by $\hat K_n(s) = \prodi_0^{s-}(1-\hat \Lambda_n(\d u))$. This is basically the Kaplan--Meier estimator, but for the censoring distribution and in a left-continuous version. Finally, if $\hat \nu_n(s) = n^{-1} \sum_{i=1}^n \tilde N_i(s)$ as before, the estimate obtained by
\begin{equation} \label{eq:ipcw_estimate}
    \hat \mu_n(s) = \int_0^s \frac{1}{\hat K_n(u)} \hat \nu_n(\d u)
\end{equation}
yields an estimate of $\mu(s)$. This estimate corresponds to the estimate of~\eqref{eq:ipcw_estimate_emp} from Example~\ref{example:cens_obs} if no censoring of the censoring times occurs before time $s$, and is in this sense a generalization. 
The estimate of~\eqref{eq:ipcw_estimate} also corresponds to the estimate studied by \cite{ghosh2000}, which is also considered in \cite{Andersen2019}.

The estimate in \eqref{eq:ipcw_estimate} relies on empirical means of three counting processes, namely $\tilde N$, $N_{X,0}$, and $N_{X,1}$, where $N_{X,j}(s) = \indic{\tilde C \leq s, \tilde D = j}$ for $j=0, 1$, through a functional which is differentiable of any order in a $p$-variation setting. This allows us to take a similar approach as in Example~\ref{example:cens_obs} when studying the asymptotic properties of the estimator. 
Specifically, we let $\banach{F} = \W_p\rc([0, \infty))^3$, for a $p \in [1,2)$, with a general element of the form $f= (f_1, f_2, f_3) \in \banach{F}$ and a norm given by $\|f\|_{\banach{F}} = \max(\|f_1\|_{[p]}, \|f_2\|_{[p]}, \|f_3\|_{[p]})$. In particular, $F := (\nu, G_1\c, G_0\c) \in \banach{F}$, where $G_0\c(s) = \P(\tilde C \leq s, \tilde D = 0)$.
With $x = (\tilde n, \tilde c, \tilde d)$ for $\tilde n \in \W_p\rc$, $\tilde c > 0$, and $\tilde d \in \{0,1\}$, define $\delta_x$ by $\delta_x(s) = (\tilde n(s), N_{x,1}(s), N_{x,0}(s))$, where $N_{x,1}(s) = \indic{\tilde c \leq s, \tilde d = 1}$ and $N_{x,0}(s) = \indic{\tilde c \leq s, \tilde d = 0})$. Based on $n$ independent replications of $X = (\tilde N, \tilde C, \tilde D)$, the empirical version of $F$ is $F_n = n^{-1} \sum_{i=1}^n \delta_{X_i}$. In the following, the estimate $\hat \mu_n(s)$ from \eqref{eq:ipcw_estimate} is studied as a functional of $F_n$. This is done for $s \in [0,t]$ for a given $t > 0$ that satisfies $K\c(t) > 0$.
Define a $K\c$ functional by $K\c(f;s) = f_2(\infty) + f_3(\infty) - f_2(s-) - f_3(s)$. This functional is continuous and linear and so differentiable of any order as a functional from $\banach{F}$ to $\W_p([0,t])$ with first order derivative given by ${K\c}_f'(g;s) = K\c(g;s)$. Since $K\c(F;s) = K\c(s)$ and $K\c(t) > 0$, a $\Lambda$ functional can, at least in a neighborhood of $F$, be defined by
\begin{equation}
    \Lambda(f;s) = \int_0^s \frac{1}{K\c(f;u)} f_2(\d u),
\end{equation}
such that $f \mapsto \Lambda(f)$ is mapping into $\W_p\rc([0,t])$.
We see that $\Lambda(F;s) = \int_0^s K\c(u)^{-1} G_1\c(\d u) = \Lambda(s)$ as well as $\Lambda(F_n;s) = \hat \Lambda_n(s)$.
By the arguments in Example~\ref{example:cens_obs}, the $\Lambda$ functional is differentiable of any order in a neighborhood of $F$ with first order derivative
\begin{equation}
    \Lambda_f'(g;s) = \int_0^s \frac{1}{K\c(f;u)} g_2(\d u) - \int_0^s \frac{K\c(g;u)}{K\c(f;u)^2} f_2(\d u).
\end{equation}
Note how $\Lambda_F'(\delta_x - F;s) = \Lambda_F'(\delta_x;s) = \int_0^s K\c(u)^{-1} M_{x,1}(\d u)$ where $M_{x,1}(s) = N_{x,1}(s) - \int_0^s \indic{\tilde c \geq u} \Lambda(\d u)$.
Next, a $K$ functional can be defined by $K(f;s) = \prodi_0^{s-}(1- \Lambda(f; \d u))$ as a functional from $\banach{F}$ to $\W_p([0,t])$. This will then satisfy $K(F;s) = K(s)$ and $K(F_n;s) = \hat K_n(s)$.
The product integral $f \mapsto \prodi_0^{(\cdot)}(1+f(\d u))$, as a functional from $\W_p\rc([0,t])$ to $\W_p\rc([0,t])$ for a $p \in [1,2)$, is differentiable of any order with a first order derivative at $f$ in direction $g$ which is $\int_0^{(\cdot)} \prodi_0^{s-}(1+f(\d u)) g(\d s) \prodi_s^{(\cdot)}(1+f(\d u))$, which when $\Delta f(s):=f(s)-f(s-) \neq -1$ for all $s \in [0,t]$ can also be given as $\prodi_0^{(\cdot)}(1+f(\d u)) \int_0^{(\cdot)} (1+\Delta f(s))^{-1} g(\d s)$.
Since $\Lambda(s) < 1$ for all $s \in [0,t]$, the chain rule now reveals that the $K$ functional is differentiable of any order in a neighborhood of $F \in \banach{F}$ with first order derivative given by
\begin{equation}
    K_f'(g;s) = -K(f;s) \int_0^{s-} \frac{1}{1-\Delta \Lambda(f;u)} \Lambda_f'(g;\d u).
\end{equation}
Using the expression of $\Lambda_F'(\delta_x - F;s)$ given above, the expression 
\begin{equation} \label{eq:K_infl}
    K_F'(\delta_x - F;s) = -K(s) \int_0^{s-} \frac{1}{1-\Delta \Lambda(u)} \frac{1}{K\c(u)} M_{x,1}(\d u)
\end{equation}
can be obtained.
Lastly, the $\mu$ functional is defined by 
\begin{equation}
    \mu(f;s) = \int_0^s \frac{1}{K(f;u)} f_1(\d u)
\end{equation}
as a functional from $\banach{F}$ to $\W_p\rc([0,t])$.
The functional satisfies $\mu(F;s) = \mu(s)$ and $\mu(F_n;s) = \hat \mu_n(s)$ for $\hat \mu_n$ from~\eqref{eq:ipcw_estimate}.
As in Example~\ref{example:cens_obs}, the $\mu$ functional is differentiable of any order in a neighborhood of $F \in \banach{F}$. The first order derivative is given by
\begin{equation} \label{eq:cens_unobs_mu_deriv}
    \mu_f'(g;s) = \int_0^s \frac{1}{K(f; u)} g_1(\d u) - \int_0^s \frac{K_f'(g;u)}{K(f;u)^2} f_1(\d u).
\end{equation}
Using the expression of $K_F'(\delta_x - F;s)$ from \eqref{eq:K_infl} the influence function $\dot \mu$ can be expressed as 
\begin{equation}
\begin{aligned}
    \dot \mu(x;s) &= \int_0^s \frac{1}{K(u)} \tilde n(\d u) - \mu(s)\\
    &\phantom{{}=} + \int_0^s \int_0^{u-} \frac{1}{1-\Delta \Lambda(v)} \frac{1}{K\c(v)} M_{x,1}(\d v) \mu(\d u)
\end{aligned}
\end{equation}
for $x = (\tilde n, \tilde c, \tilde d)$. 
As was the case in Example~\ref{example:cens_obs}, using $p \in (4/3,2)$ allows for the conclusion that, for $s \in [0,t]$, $\sqrt{n}(\hat \mu_n(s) - \mu(s))$ has an asymptotic normal distribution with mean $\E(\dot \mu(X;s)) = 0$ and a variance of $\Var(\dot \mu(X;s))$.
In terms of the potentially unobserved $N$, $T$ and $C$, the influence function at $X$ can also be expressed as
\begin{equation} \label{eq:C_unobs_infl}
\begin{aligned}
    \dot \mu(X;s) &= N(s) - \mu(s) \\
    & \hspace{-12pt}- \int_0^{s-} \big(N(s) - N(u) - \frac{\indic{T > u}}{\P(T > u)} (\mu(s) - \mu(u)\big)\frac{1}{K(u+)} M_C(\d u).
\end{aligned}
\end{equation}
This expression leads to the variance expression
\begin{equation} \label{eq:var_expr_C_unobs}
\begin{aligned}
   \Var(\dot \mu(X;s)) &= \Var(N(s)) + \int_0^{s-} \Var(N(s) - N(u)) \frac{1}{K(u+)} \Lambda(\d u) \\
   &\phantom{{}=} - \int_0^{s-} (\mu(s) - \mu(u))^2 \frac{\P(T \leq u)}{\P(T > u)} \frac{1}{K(u+)} \Lambda(\d u)
  \end{aligned}
\end{equation}
under the independence assumption $(N,T) \independent C$.
This variance can be estimated by $n^{-1} \sum_{i=1}^n \mu_{F_n}'(\delta_{X_i} - F_n;s)^2$ where the expression of $\mu_{F_n}'(\delta_x - F_n; s)$ can be obtained by insertion in \eqref{eq:cens_unobs_mu_deriv}. This variance estimate will be very similar to the one suggested by \cite{ghosh2000} and seemingly identical in the absence of ties.
Some more details on the derivations of~\eqref{eq:C_unobs_infl} and~\eqref{eq:var_expr_C_unobs} can be found in Appendix~\ref{appendix:infl_var}. In comparison to Example~\ref{example:cens_obs} where the actual censoring times are available, the last term of~\eqref{eq:var_expr_C_unobs} reveals that this asymptotic variance is smaller than for the estimator of Example~\ref{example:cens_obs}. 
This means that even when information is available on the potential censoring times $C_1, \dots, C_n$ in a setting with a terminal event and from an asymptotic point of view, the analyst is better off by disregarding this complete information and relying only on the censored censoring times.
\end{example}

\begin{example} \label{example:pseudo}
The pseudo-observation method is a method for regression analysis of an outcome such as $N(t)$ when the outcomes are incompletely observed such as in examples~\ref{example:cens_obs} and~\ref{example:cens_unobs}. Given $n$ independent replications $X_1, \dots, X_n$ of $X$, the method works by substituting all the potentially unobserved outcomes $N_1(t), \dots, N_n(t)$ for jack-knife pseudo-values, $\hat \mu_{n,1}(t), \dots, \hat \mu_{n,n}(t)$ with $\hat \mu_{n,i}(t) = n \hat \mu_n(t) - (n-1) \hat \mu_n^{(i)}(t)$, and proceeding by performing whatever regression analysis was intended for $N_1(t), \dots, N_n(t)$. Here, $\hat \mu_n(t)$ is an estimator of the expectation $\mu(t) = \E(N(t))$ based on the sample $X_1, \dots, X_n$ and $\hat \mu_n^{(i)}(t)$ is the same estimator applied to the sample where the $i$th observation has been left out.
Suppose $Z$ denotes covariates and the regression analysis concerns a model of $\E(N(t) \given Z)$. According to \cite{overgaard2017asymptotic}, this pseudo-observation approach will work, under some regularity conditions, in a setting where the estimator can be seen as a functional applied to a sample average, $\hat \mu_n(t) = \mu(F_n;t)$ for a functional $\mu(\cdot; t) \from \banach{F} \to \R$ defined on a Banach space $ \banach{F}$ where $F_n = n^{-1} \sum_{i=1}^n \delta_{X_i}$ for some function $x \mapsto \delta_x$ applied to the observed $X_1, \dots, X_n$, if 
\begin{enumerate}[label=(\alph*)]
    \item \label{it:conv} an $F \in \banach{F}$ and an $\varepsilon \in (0, 1/4]$ exist such that $\|F_n - F\|_{\banach{F}} = o_{\P}(n^{-1/4 - \varepsilon/2})$ and $\lim_{y \to \infty} y^{1/\varepsilon}\P(\|\delta_X\| > y) = 0$, 
    \item \label{it:diff} the functional $f \mapsto \mu(f;t)$ is continuously differentiable of order 2 with a Lipschitz continuous second order derivative in a neighborhood of $F$,
    \item \label{it:cond} the influence function $\dot \mu(x;t) = \mu_F'(\delta_x - F;t)$ satisfies 
    \begin{equation}
        \E(\dot \mu(X;t) \given Z) = \E(N(t) \given Z) - \mu(F;t).
    \end{equation}
\end{enumerate}
Conditions~\ref{it:conv} and~\ref{it:diff} agree well with the estimators of examples~\ref{example:cens_obs} and~\ref{example:cens_unobs} above.
The condition that $\lim_{y \to \infty} y^{1/\varepsilon}\P(\|\delta_X\| > y) = 0$ is fulfilled in either example if $N(t)$ has finite moment of order a little higher than $1/\varepsilon$, that is, at least a little more than fourth order with the choice $\varepsilon = 1/4$. The convergence order $\|F_n - F\|_{\banach{F}} = o_{\P}(n^{-3/8})$ with this choice of $\varepsilon$ is achieved for the $p$-variation-based norm for $p \in (8/5,2)$.
The range is $p \in (4/(3-2 \varepsilon), 2)$ more generally for the relevant $\varepsilon \in (0,1/4]$.
The functionals involved in the estimators of the examples are differentiable of any order for such choices of $p$, and so condition~\ref{it:diff} is met in this setting since the Lipschitz continuity of the second order derivative in a neighborhood of $F$ follows from third order differentiability in a neighborhood of $F$.

This leaves us with condition~\ref{it:cond}. If we assume $(N,Z) \independent C$ or $(N,T,Z) \independent C$ in the examples, respectively, this can easily be seen to be fulfilled by appealing to the expressions of~\eqref{eq:C_obs_infl} and~\eqref{eq:C_unobs_infl}, respectively, since $\E(M_C(s) \given N, Z) = 0$ or similarly $\E(M_C(s) \given N, T, Z) = 0$ for $s \in [0,t]$ in these settings.
The pseudo-observation method with pseudo-observations based on the estimator of Example~\ref{example:cens_unobs} has been suggested and applied by Andersen, Angst, and Ravn in \cite{Andersen2019}, where their equation~(4) corresponds to~\eqref{eq:ipcw_estimate} of this paper.
With conditions \ref{it:conv}, \ref{it:diff}, and \ref{it:cond} fulfilled, the results of \cite{overgaard2017asymptotic} now brings a theoretical justification to this approach.
\end{example}

\section{Concluding remarks}

In many cases counting processes may go to infinity as time passes. Such a setting does not fit well with an assumption of finite $p$-variation or finite moment conditions and the approach described in this paper is not directly applicable. It may be useful to consider stopped or localized versions of such counting processes, and such stopped or localized counting processes may perhaps be studied using the $p$-variation approach described here.
Concretely, the settings of Examples~\ref{example:uncens}--\ref{example:pseudo} reduced interest to the interval $[0,t]$ or simply $t$ for some time point $t > 0$ and the stopped processes $N(\cdot \wedge t)$ can replace $N$ without issues in such cases.

The convergence results in $p$-variation of Section~\ref{sec:main} are likely not the best possible and further studies of this subject are called for. It is, for instance, not clear whether the moment condition of Theorem~\ref{theorem:generalE} or the boundedness condition of Theorem~\ref{theorem:almost_sure} are necessary or if such convergence results apply more generally to averages of independent replications of random elements in $\W_p$ spaces.

In Banach spaces, measurability is not a straightforward matter. Measurability has not been touched upon in any detail here. When a counting process $N$ is considered a random element, it is in the sense of measurable coordinate projections, that $N(s)$ is a random variable for all relevant $s$. In the examples of Section~\ref{sec:count}, the various $\mu$ functionals have not been formalized as measurable maps from $\banach{F}$ to $\W_p\rc$. It is however clear by inspection that at $s \in [0,t]$, the various $\hat \mu_n(s)$ and $\dot \mu(X;s)$ are random variables. 
Owing to right-continuity, the $\|F_n - F\|_{[p]}$ of Section~\ref{sec:main}, and so similarly the various $\|F_n - F\|_{\banach{F}}$ of Section~\ref{sec:count}, are random variables since the involved suprema can be taken over the rationals.

Various convergence results exist in $p$-variation for $p \geq 2$, see for instance Theorem~3.1 and~4.1 of \cite{qian1998} and Theorem~1 of \cite{huang2001}. 
Since somewhat fewer functionals are differentiable in a $p$-variation setting for such $p$s, this has not been considered here.

\section*{Acknowledgement}
Comments and suggestions by Erik Thorlund Parner and Jan Pedersen have considerably improved this paper. 
The work presented in this article is supported by the Novo Nordisk Foundation, grant NNF17OC0028276.

\appendix
\section{Fréchet differentiability} \label{appendix:differentiability}

A functional $\phi$ defined on an open subset $U$ of a Banach space $\banach{D}$ and with values in a Banach space $\banach{E}$ is said to be differentiable at $f \in U$ if a linear continuous operator $\phi_f' \in L(\banach{D}, \banach{E})$ exists such that 
\begin{equation}
    \|\phi(f + g) - \phi(f) - \phi_f'(g)\|_{\banach{E}} = o(\|g\|_{\banach{D}})
\end{equation}
as $\|g\|_{\banach{D}} \to 0$. In that case, $\phi_f' \in L(\banach{D}, \banach{E})$ is the first order derivative of $\phi$ at $f$ and, for any $g \in \banach{D}$, $\phi_f'(g) \in \banach{E}$ is called the first order derivative of $\phi$ at $f$ in direction $g$.
The space of linear, continuous operators $L(\banach{D}, \banach{E})$ is itself a Banach space when equipped with the operator norm given by
\begin{equation}
    \|\lambda\|_{L(\banach{D}, \banach{E})} = \inf \{c \geq 0 : \|\lambda(f)\|_{\banach{E}} \leq c \|f \|_{\banach{D}} \textup{ for all } f \in \banach{D}\}
\end{equation}
and the first order derivative $\phi' \from U \to L(\banach{D}, \banach{E})$, given by $f \mapsto \phi_f'$, is simply another functional. Higher order differentiability of the functional $\phi$ can then iteratively be defined in terms of differentiability properties of the functional $\phi'$.
If $\phi$ is differentiable of order $k$, the $k$th order derivative can be identified with a functional $\phi^{(k)} \from U \to L^k(\banach{D}, \banach{E})$ where $L^k(\banach{D}, \banach{E})$ is the space of $k$-linear, continuous operators. For $\lambda \in L^k(\banach{D}, \banach{E})$ a $c > 0$ exists such that
\begin{equation}
    \|\lambda(f_1, \dots, f_k)\|_{\banach{E}} \leq c \|f_1\|_{\banach{D}} \dots \|f_k\|_{\banach{D}}
\end{equation}
and, in similarity to the $k=1$ case, the norm of $\lambda$ in the Banach space $L^k(\banach{D}, \banach{E})$ is given by the infimum over such constants $c$. The $k$th order derivative is not only continuous and $k$-linear, but also symmetric in its arguments.
If $\phi \from U \to \banach{E}$, in the setting from before, is continuously differentiable of order $k$, a $k$th order Taylor approximation in line with
\begin{equation}
    \phi(f+g) = \phi(f) + \sum_{j=1}^k \frac{1}{j!} \phi_f^{(j)}(g, \dots, g) + o(\|g\|_{\banach{D}}^k)
\end{equation}
applies as $\|g\|_{\banach{D}} \to 0$. 
In particular if $\phi$ is continuously differentiable of order 2 in a neighborhood of $f \in \banach{D}$, or weaker still if $\phi'$ is Lipschitz continuous in a neighborhood of $f \in \banach{D}$, then 
\begin{equation} \label{eq:Taylor1+Lip}
    \phi(f+g) = \phi(f) + \phi_f'(g) + O(\|g\|_{\banach{D}}^2)
\end{equation}
as $\|g\|_{\banach{D}} \to 0$.
If $\banach{D}$, $\banach{E}$, and $\banach{F}$ are Banach spaces and $\phi$ is a functional defined and differentiable on a neighborhood of $f \in \banach{D}$ as a functional into $\banach{E}$, whereas $\psi$ is a functional defined and differentiable on a neighborhood of $\phi(f) \in \banach{E}$ as a functional into $\banach{F}$, then $\psi \circ \phi$ is differentiable in a neighborhood of $f \in \banach{D}$ as a functional into $\banach{F}$. The derivative is given by 
\begin{equation} \label{eq:chain_rule}
    (\psi \circ \phi)_f'(g) = \psi_{\phi(f)}'(\phi_f'(g)),
\end{equation}
the derivative of $\psi$ at $\phi(f)$ in direction $\phi_f'(g)$.
This is the chain rule.

\section{Influence functions and the variance expressions} \label{appendix:infl_var}
In obtaining the desired expressions of the influence functions and variance expressions in examples~\ref{example:cens_obs} and~\ref{example:cens_unobs}, an important identity is
\begin{equation}
    \frac{\indic{C \geq s)}}{K(s)} - 1 = -\int_0^{s-} \frac{1}{K(u+)} M_C(\d u),
\end{equation}
which can be seen as a consequence of the Duhamel equation, see for instance \cite{Gill1990}, here in the form
\begin{equation}
    \indic{C \geq s} - K(s) = \int_0^{s-} \indic{C \geq u}(\Lambda - N_C)(\d u) \frac{K(s)}{K(u+)}.
\end{equation}
Since $\tilde N(s) = \int_0^s \indic{C \geq u} N(\d u)$, we have in Example~\ref{example:cens_obs}
\begin{equation}
\begin{aligned}
    &\int_0^s \frac{1}{K(u)} \tilde N(\d u) - \int_0^s \frac{\indic{C \geq u}}{K(u)} \mu(\d u) \\
    =& \int_0^s \frac{\indic{C \geq u}}{K(u)}(N - \mu)(\d u) \\
    =& N(s) - \mu(s) + \int_0^s \big(\frac{\indic{C \geq u}}{K(u)} - 1\big)(N - \mu)(\d u) \\
    =& N(s) - \mu(s) - \int_0^s \int_0^{u-} \frac{1}{K(v+)} M_C(\d v)(N - \mu)(\d u) \\
    =& N(s) - \mu(s) - \int_0^{s-} \frac{N(s) - \mu(s) - N(v) + \mu(v)}{K(v+)} M_C(\d v), 
\end{aligned}
\end{equation}
which shows the alternative expression of the influence function in Example~\ref{example:cens_obs}.
For Example~\ref{example:cens_unobs}, it can be noted that we have $M_{X,1}(s) = \int_0^s \indic{T > u} M_C(\d u)$ and also $(1- \Delta \Lambda(s))K\c(s) = \P(T > s) K(s+)$. It is now a similar argument as above which yields
\begin{equation}
\begin{aligned}
    &\int_0^s \frac{1}{K(u)} \tilde N(\d u) - \mu(s) + \int_0^s \int_0^{u-} \frac{1}{1-\Delta \Lambda(v)} \frac{1}{K\c(v)} M_{X,1}(\d v) \mu(\d u) \\
    =& N(s) - \mu(s) + \int_0^s \big(\frac{\indic{C \geq u}}{K(u)} - 1\big) N(\d u) \\
    &+ \int_0^s \int_0^{u-} \frac{\indic{T > v}}{\P(T > v)} \frac{1}{K(v+)} M_{C}(\d v) \mu(\d u)  \\
    =& N(s) - \mu(s) \\
    &- \int_0^{s-} \big(N(s) - N(v) - \frac{\indic{T > v}}{\P(T > v)}(\mu(s) - \mu(v))\big) \frac{1}{K(v+)} M_{C}(\d v).
\end{aligned}
\end{equation}
The process given by $M_C(s) = N_C(s) - \int_0^s \indic{C \geq u} \Lambda(\d u)$ is a martingale with respect to the natural filtration of $N_C$. If we assume $N \independent C$ or $(N,T) \independent C$ this is the case even in the conditional distribution given $N$ or given $(N,T)$. 
So, for certain stochastic processes $A(s)$ and $B(s)$ that are measurable with respect to $\sigAlg{A} = \sigma(N)$ or $\sigAlg{A} = \sigma(N, T)$, we have $\Var(A(s) + \int_0^s B(u) K(u+)^{-1}M_{C}(\d u) \given \sigAlg{A}) = \int_0^s B(u)^2 K(u+)^{-1} \Lambda(\d u)$ by martingale properties since, for instance, the optional variation process of $M_C$ is given by $[M_C](s) = \int_0^s (1-\Delta \Lambda(u)) N_C(\d u) - \int_0^s \Delta \Lambda(u) M_C(s)$ with conditional expectation $\E([M_C](s) \given \sigAlg{A} ) = \int_0^s (1- \Delta \Lambda(u)) G(\d u) = \int_0^s K(u+) \Lambda(\d u)$.
The law of total variation then reveals $\Var(A(s) + \int_0^s B(u) K(u+)^{-1} M_C(\d u)) = \Var(A(s)) + \int_0^s \E(B(u)^2) K(u+)^{-1} \Lambda(\d u)$. Both~\eqref{eq:var_expr_C_obs} and~\eqref{eq:var_expr_C_unobs} follow this structure, although establishing~\eqref{eq:var_expr_C_unobs} requires an additional direct calculation as follows.
In this case, $N(s) - N(u) - \indic{T > u} \P(T > u)^{-1}(\mu(s) - \mu(u))$ is playing the role of $B(u)$. It is the fact that $T$ is terminal for $N$ which implies $(N(s) - N(u)) \indic{T > u} = N(s) - N(u)$ and so
\begin{equation}
    \begin{aligned}
    &\phantom{{}=}\E\big(\big(N(s) - N(u) - \frac{\indic{T > u}}{\P(T > u)}(\mu(s) - \mu(u)\big)^2\big) \\
    &=\E((N(s)-N(u))^2) - 2 \frac{\E(N(s) - N(u))}{\P(T > u)} (\mu(s) - \mu(u)) \\
    &\phantom{{}=}+ \frac{\E(\indic{T > u})}{\P(T > u)^2}(\mu(s) - \mu(u))^2 \\
    &= \E((N(s)-N(u))^2) - \frac{1}{\P(T > u)}(\mu(s) - \mu(u))^2 \\
    &= \Var(N(s) - N(u)) - \frac{\P(T \leq u)}{\P(T > u)}(\mu(s) - \mu(u))^2
    \end{aligned}
\end{equation}
as desired.

\end{document}